\documentclass[11pt]{amsart}%
\usepackage{graphicx}
\usepackage{amscd, color}
\usepackage{amsmath}
\usepackage{amsfonts}
\usepackage{amssymb}%
\providecommand{\U}[1]{\protect\rule{.1in}{.1in}}
\providecommand{\U}[1]{\protect\rule{.1in}{.1in}}
\providecommand{\U}[1]{\protect\rule{.1in}{.1in}} \textwidth 17cm
\theoremstyle{plain}

\newtheorem{theorem}{Theorem}[section]

\newtheorem{lemma}[theorem]{Lemma}
\newtheorem{problem}[theorem]{Problem}

\numberwithin{equation}{section}

\begin{document}
\title{On summability of nonlinear mappings: a new approach}
\author{Daniel Pellegrino and Joedson Santos}
\address[D. Pellegrino]{ Departamento de Matem\'{a}tica, Universidade Federal da
Para\'{\i}ba, 58.051-900 - Jo\~{a}o Pessoa, Brazil\\
[J. Santos] Departamento de Matem\'{a}tica, Universidade Federal de Sergipe,
49500-000-Itabaiana, Brazil}
\email{dmpellegrino@gmail.com}
\thanks{D. Pellegrino by INCT-Matem\'{a}tica, PROCAD-NF Capes, CNPq Grant
620108/2008-8 (Ed. Casadinho) and CNPq Grant 301237/2009-3. }

\begin{abstract}
The main goal of this paper is to characterize arbitrary nonlinear
(non-multilinear) mappings $f:X_{1}\times\cdots\times X_{n}\rightarrow Y$
between Banach spaces that satisfy a quite natural Pietsch Domination-type
theorem around a given point $(a_{1},...,a_{n})\in$ $X_{1}\times\cdots\times
X_{n}$. As a consequence of our approach a notion of weighted summability
arises naturally, which may be an interesting topic for further investigation.

\end{abstract}
\maketitle

\section{Introduction}

The theory of absolutely summing operators was initiated with Grothendieck%
\'{}%
s ideas in the 50s but just in the sixties (see \cite{LP, stu}) the results
were better understood and fully explored (for details we refer to the book
\cite{DJT}). Besides its intrinsic interest, this theory has beautiful
applications in Banach space theory and nice connections with the geometry of
the Banach spaces involved (see, for example, \cite{DavisJ, LP} or
\cite{PellZ} for a more recent approach). Due to the success of the linear
theory, it is not a surprise that many authors have devoted their interest to
the nonlinear setting; the multilinear theory, however, has a longer history,
which seems to start with \cite{BH, LLL}; for recent different nonlinear
approaches and applications we mention \cite{DPE, Def, Def2, Popa, Junek,
Nach, MP, SST, ppp, Perr} and references therein.

Pietsch Domination-Factorization Theorems play a central role in the theory of
absolutely summing linear operators and provide an unexpected and beautiful
measure theoretic taste in the theory (for details we mention the monographs
\cite{AK, DF, DJT, Ryan}). In the last decade several different nonlinear
versions of Pietsch Domination-Factorization Theorem have\ appeared in the
literature (see, for example, \cite{AMe, JFA, BPR, FaJo, Geiss, SP}); for this
reason, in \cite{BPRn}, an abstract unified approach to Pietsch-type results
was presented as an attempt to show that all the known Pietsch-type theorems
were particular cases of a unified general version. The main problem
investigated in the present paper is motivated by the Pietsch- Domination
Theorem (PDT) for $n$-linear mappings between Banach spaces, which we describe below.

From now on, if $X_{1},...,X_{n},Y$ are Banach spaces over a fixed scalar
field which can be either $\mathbb{K}=\mathbb{R}$ or $\mathbb{C}$,
$Map(X_{1},...,X_{n};Y)$ will denote the set of all arbitrary mappings from
$X_{1}\times\cdots\times X_{n}$ to $Y$ (no assumption is necessary). The
topological dual of a Banach space $X$ will be denoted by $X^{\ast}$ and its
closed unit ball will be represented by $B_{X^{\ast}},$ with the weak-star topology.

Let $0<p_{1},...,p_{n}<\infty$ and $1/p=%
{\textstyle\sum\limits_{j=1}^{n}}
1/p_{j}.$ An $n$-linear mapping $T:X_{1}\times\cdots\times X_{n}\rightarrow Y$
is $(p_{1},...,p_{n})$-dominated if there is a constant $C>0$ so that
\begin{equation}
\left(  \sum_{j=1}^{m}\left\Vert T(x_{j}^{(1)},...,x_{j}^{(n)})\right\Vert
^{p}\right)  ^{\frac{1}{p}}\leq C%
{\displaystyle\prod\limits_{k=1}^{n}}
\sup_{\varphi\in B_{X_{k}^{\ast}}}\left(
{\displaystyle\sum\limits_{j=1}^{m}}
\left\vert \varphi(x_{j}^{(k)})\right\vert ^{p_{k}}\right)  ^{1/p_{k}},
\label{abs}%
\end{equation}
regardless of the choice of the positive integer $m,$ $x_{j}^{(k)}\in X_{k}$,
$k=1,...,n$ and $j=1,...,m$. The folkloric PDT for $(p_{1},...,p_{n}%
)$-dominated multilinear mappings (see \cite{Geiss} or \cite{Perr2} for a
detailed proof) asserts that $T$ is $(p_{1},...,p_{n})$-dominated if and only
if there are Borel probabilities $\mu_{k}$ on $B_{X_{k}^{\ast}},$ $k=1,...,n$,
and a constant $C>0$ such that%
\begin{equation}
\left\Vert T(x^{(1)},...,x^{(n)})\right\Vert \leq C\left(
{\displaystyle\int\nolimits_{B_{X_{1}^{\ast}}}}
\left\vert \varphi(x^{(1)})\right\vert ^{p_{1}}d\mu_{1}\right)  ^{\frac
{1}{p_{1}}}\cdots\left(
{\displaystyle\int\nolimits_{B_{X_{n}^{\ast}}}}
\left\vert \varphi(x^{(n)})\right\vert ^{p_{n}}d\mu_{k}\right)  ^{\frac
{1}{p_{n}}} \label{tyy}%
\end{equation}
for all $x^{(j)}\in X_{j}$, $j=1,...,n.$

A related question, not covered by the abstract approach presented in
\cite{BPRn}, arises:

\begin{problem}
\label{wee}If $(a_{1},...,a_{n})\in X_{1}\times\cdots\times X_{n}$, what kind
of mappings $f\in Map(X_{1},...,X_{n};Y)$ satisfy, for some $C>0$ and Borel
probabilities $\mu_{k}$ on $B_{X_{k}^{\ast}},$ $k=1,...,n$, the inequality%
\begin{equation}
\left\Vert f(a_{1}+x^{(1)},...,a_{n}+x^{(n)})-f(a_{1},...,a_{n})\right\Vert
\leq C%
{\displaystyle\prod\limits_{k=1}^{n}}
\left(  \int_{B_{X_{k}^{\ast}}}\left\vert \varphi(x^{(k)})\right\vert ^{p_{k}%
}d\mu_{k}\right)  ^{\frac{1}{p_{k}}} \label{domGG}%
\end{equation}
for all $x^{(j)}\in X_{j}$, $j=1,...,n$ $?$
\end{problem}

In the next section we solve Problem \ref{wee}.

\section{Main Result}

Let $0<p_{1},...,p_{n}<\infty$ and $1/p=%
{\textstyle\sum\limits_{j=1}^{n}}
1/p_{j}.$ We will say that $f\in Map(X_{1},...,X_{n};Y)$ is $(p_{1}%
,...,p_{n})$-dominated at $(a_{1},...,a_{n})\in X_{1}\times\cdots\times X_{n}$
if there is a $C>0$ and there are Borel probabilities $\mu_{k}$ on
$B_{X_{k}^{\ast}},$ $k=1,...,n$, such that (\ref{domGG}) is valid for all
$x^{(j)}\in X_{j}$, $j=1,...,n$.

It is worth mentioning that Pietsch's original proof of his domination theorem
uses Ky Fan Lemma instead of the usual Hahn-Banach separation theorem (see
\cite{Pi}). The use of Hahn-Banach theorem seems to be not adequate for
proving our main result; for this task Pietsch's original idea of using Ky Fan
Lemma will be very useful. It is in some sense a nice surprise that Pietsch's
first argument conceived for linear maps has shown to be the more adequate
when dealing with a very general and fully nonlinear context.

\begin{lemma}
[Ky Fan]Let $K$ be a compact Hausdorff topological space and $\mathcal{F}$ be
a concave family of functions $f:K\rightarrow\mathbb{R}$ which are convex and
lower semicontinuous. If for each $f\in\mathcal{F}$ there is a $x_{f}\in K$ so
that $f(x_{f})\leq0,$ then there is a $x_{0}\in K$ such that $f(x_{0})\leq0$
for every $f\in\mathcal{F}$ .
\end{lemma}

For the proof of our main theorem we will need the following lemma (see
\cite[Page 17]{Hardy}):

\begin{lemma}
\label{yy}Let $0<p_{1},...,p_{n},p<\infty$ be so that $1/p=%
{\textstyle\sum\limits_{j=1}^{n}}
1/p_{j}$. Then%
\[
\frac{1}{p}%
{\displaystyle\prod\limits_{j=1}^{n}}
q_{j}^{p}\leq%
{\displaystyle\sum\limits_{j=1}^{n}}
\frac{1}{p_{j}}q_{j}^{p_{j}}%
\]
regardless of the choices of $q_{1},..,q_{n}\geq0.$
\end{lemma}

\begin{theorem}
\label{ttta}A map $f\in Map(X_{1},...,X_{n};Y)$ is $(p_{1},...,p_{n}%
)$-dominated at $(a_{1},...,a_{n})\in X_{1}\times\cdots\times X_{n}$ if and
only if there is a $C>0$ such that
\begin{align}
&  \left(
{\displaystyle\sum\limits_{j=1}^{m}}
\left(  \left\vert b_{j}^{(1)}...b_{j}^{(n)}\right\vert \left\Vert
f(a_{1}+x_{j}^{(1)},...,a_{n}+x_{j}^{(n)})-f(a_{1},...,a_{n})\right\Vert
\right)  ^{p}\right)  ^{1/p}\label{qww}\\
&  \leq C%
{\displaystyle\prod\limits_{k=1}^{n}}
\sup_{\varphi\in B_{X_{k}^{\ast}}}\left(
{\displaystyle\sum\limits_{j=1}^{m}}
\left(  \left\vert b_{j}^{(k)}\right\vert \left\vert \varphi(x_{j}%
^{(k)})\right\vert \right)  ^{p_{k}}\right)  ^{1/p_{k}}\nonumber
\end{align}
for every positive integer $m$, $(x_{j}^{(k)},b_{j}^{(k)})\in X_{k}%
\times\mathbb{K}$, with $(j,k)\in\{1,...,m\}\times\{1,...,n\}.$
\end{theorem}

\begin{proof}
In order to simplify notation, from now on we will write%
\[
f(b^{(k)},x^{(k)})_{k=1}^{n}:=\left(  \left\vert b^{(1)}...b^{(n)}\right\vert
\left\Vert f(a_{1}+x^{(1)},...,a_{n}+x^{(n)})-f(a_{1},...,a_{n})\right\Vert
\right)  ^{p}.
\]

Assume the existence of such measures $\mu_{1},...,\mu_{n}$ satisfying
(\ref{domGG}). Then, given $m\in\mathbb{N}$, $x_{j}^{(l)}\in E_{l}$ and
$b_{j}^{(l)}\in\mathbb{K},$ with $(j,l)\in\{1,...,m\}\times\{1,...,n\},$ we
have, using H\"{o}lder Inequality,
\begin{align*}
&  \sum_{j=1}^{m}f(b_{j}^{(k)},x_{j}^{(k)})_{k=1}^{n}\leq C^{p}\sum_{j=1}^{m}%
{\displaystyle\prod\limits_{k=1}^{n}}
\left(  \int_{B_{X_{k}^{\ast}}}\left(  \left\vert b_{j}^{(k)}\right\vert
\left\vert \varphi(x_{j}^{(k)})\right\vert \right)  ^{p_{k}}d\mu_{k}\right)
^{\frac{p}{p_{k}}}\\
&  \leq C^{p}%
{\displaystyle\prod\limits_{k=1}^{n}}
\left(  \sum_{j=1}^{m}\int_{B_{X_{k}^{\ast}}}\left(  \left\vert b_{j}%
^{(k)}\right\vert \left\vert \varphi(x_{j}^{(k)})\right\vert \right)  ^{p_{k}%
}d\mu_{k}\right)  ^{\frac{p}{p_{k}}}\\
&  \leq C^{p}%
{\displaystyle\prod\limits_{k=1}^{n}}
\left(  \sup_{\varphi\in B_{X_{k}^{\ast}}}\sum_{j=1}^{m}\left(  \left\vert
b_{j}^{(k)}\right\vert \left\vert \varphi(x_{j}^{(k)})\right\vert \right)
^{p_{k}}\right)  ^{\frac{p}{p_{k}}}.
\end{align*}
Hence we have (\ref{qww}). Conversely, suppose (\ref{qww}) and consider the
sets $P(B_{X_{k}^{\ast}})$ of the probability measures in $C(B_{X_{k}^{\ast}%
})^{\ast}$, for all $k=1,...,n$. It is well-known that each $P(B_{X_{k}^{\ast
}})$ is compact when each $C(B_{X_{k}^{\ast}})^{\ast}$ is endowed with the
weak-star topology. For each $(x_{j}^{(l)})_{j=1}^{m}$ in $E_{l}$ and
$(b_{j}^{(s)})_{j=1}^{m}$ in $\mathbb{K},$ with $(s,l)\in\{1,...,n\}\times
\{1,...,n\},$ let%
\begin{align*}
&  g=g_{(x_{j}^{(l)})_{j=1}^{m},(b_{j}^{(s)})_{j=1}^{m},(s,l)\in
\{1,...,n\}\times\{1,...,n\}}:P(B_{X_{1}^{\ast}})\times\cdots\times
P(B_{X_{n}^{\ast}})\rightarrow\mathbb{R}\\
g\left(  (\mu_{i})_{i=1}^{n}\right)   &  =\sum_{j=1}^{m}\left[  \frac{1}%
{p}f(b_{j}^{(k)},x_{j}^{(k)})_{k=1}^{n}-C^{p}\sum_{k=1}^{n}\frac{1}{p_{k}}%
\int_{B_{X_{k}^{\ast}}}\left(  \left\vert b_{j}^{(k)}\right\vert \left\vert
\varphi(x_{j}^{(k)})\right\vert \right)  ^{p_{k}}d\mu_{k}\right]  .
\end{align*}
Note that the family $\mathcal{F}$ of all such $g$ is concave. In fact, let
$N$ be a positive integer, $g_{k}\in\mathcal{F}$ and $\alpha_{k}\geq0,$
$k=1,...,N,$ so that $\alpha_{1}+...+\alpha_{N}=1$. We have%
\[%
{\displaystyle\sum\limits_{k=1}^{N}}
\alpha_{k}g_{k}\left(  (\mu_{i})_{i=1}^{n}\right)  \leq g_{(x_{j_{k}}%
^{(l)})_{j_{k},k=1}^{m_{k},N},(\alpha_{k}^{\frac{1}{p_{s}}}b_{j_{k}}%
^{(s)})_{j_{k},k=1}^{m_{k},N},(s,l)\in\{1,...,n\}\times\{1,...,n\}}\left(
(\mu_{i})_{i=1}^{n}\right)  .
\]

One can also easily prove that each $g\in\mathcal{F}$ is convex and
continuous. Besides, for each $g\in\mathcal{F}$ there are measures $\mu
_{k}^{g}\in P(B_{X_{k}^{\ast}}),$ $k=1,...,n$, so that%
\[
g(\mu_{1}^{g},...,\mu_{n}^{g})\leq0.
\]
In fact, since each $B_{X_{k}^{\ast}}$ is compact ($k=1,...,n$) there are
$\varphi_{k}\in B_{X_{k}^{\ast}}$ so that%
\[
\sum_{j=1}^{m}\left(  \left\vert b_{j}^{(k)}\right\vert \left\vert \varphi
_{k}(x_{j}^{(k)})\right\vert \right)  ^{p_{k}}=\sup_{\varphi\in B_{X_{k}%
^{\ast}}}\sum_{j=1}^{m}\left(  \left\vert b_{j}^{(k)}\right\vert \left\vert
\varphi(x_{j}^{(k)})\right\vert \right)  ^{p_{k}}.
\]
Now, consider the Dirac measures $\mu_{k}^{g}=\delta_{\varphi_{k}},$
$k=1,...,n,$ and hence%
\begin{align*}
&  g(\mu_{1}^{g},...,\mu_{n}^{g})=\sum_{j=1}^{m}\left[  \frac{1}{p}%
f(b_{j}^{(k)},x_{j}^{(k)})_{k=1}^{n}\right]  -C^{p}\sum_{k=1}^{n}\frac
{1}{p_{k}}\int_{B_{X_{k}^{\ast}}}\sum_{j=1}^{m}\left(  \left\vert b_{j}%
^{(k)}\right\vert \left\vert \varphi(x_{j}^{(k)})\right\vert \right)  ^{p_{k}%
}d\mu_{k}^{g}\\
&  =\sum_{j=1}^{m}\left[  \frac{1}{p}f(b_{j}^{(k)},x_{j}^{(k)})_{k=1}%
^{n}\right]  -C^{p}\sum_{k=1}^{n}\frac{1}{p_{k}}\left[  \left(  \sup
_{\varphi\in B_{X_{k}^{\ast}}}\sum_{j=1}^{m}\left(  \left\vert b_{j}%
^{(k)}\right\vert \left\vert \varphi(x_{j}^{(k)})\right\vert \right)  ^{p_{k}%
}\right)  ^{\frac{1}{p_{k}}}\right]  ^{p_{k}}\\
&  \overset{\text{(*)}}{\leq}\sum_{j=1}^{m}\left[  \frac{1}{p}f(b_{j}%
^{(k)},x_{j}^{(k)})_{k=1}^{n}\right]  -C^{p}\frac{1}{p}%
{\displaystyle\prod\limits_{k=1}^{n}}
\left[  \left(  \sup_{\varphi\in B_{X_{k}^{\ast}}}\sum_{j=1}^{m}\left(
\left\vert b_{j}^{(k)}\right\vert \left\vert \varphi(x_{j}^{(k)})\right\vert
\right)  ^{p_{k}}\right)  ^{\frac{1}{p_{k}}}\right]  ^{p}\\
&  \overset{\text{(**)}}{\leq}0,
\end{align*}
where in (*) we have used Lemma \ref{yy} and in (**) we invoked (\ref{qww}).
So Ky Fan Lemma applies and we obtain $\overline{\mu_{k}}\in P(B_{X_{k}^{\ast
}}),$ $k=1,...,n,$ so that%
\[
g(\overline{\mu_{1}},...,\overline{\mu_{n}})\leq0
\]
for all $g\in\mathcal{F}$. Hence
\[
\sum_{j=1}^{m}\left[  \frac{1}{p}f(b_{j}^{(k)},x_{j}^{(k)})_{k=1}^{n}\right]
-C^{p}\sum_{k=1}^{n}\frac{1}{p_{k}}\int_{B_{X_{k}^{\ast}}}\sum_{j=1}%
^{m}\left(  \left\vert b_{j}^{(k)}\right\vert \left\vert \varphi(x_{j}%
^{(k)})\right\vert \right)  ^{p_{k}}d\overline{\mu_{k}}\leq0
\]
and making $m=1$ we get (for every $b^{(k)}\in\mathbb{K}$ and $x^{(k)}\in
X_{k}$, $k=1,...,n$)%
\begin{align}
&  \frac{1}{p}\left(  \left\vert b^{(1)}...b^{(n)}\right\vert \left\Vert
f(a_{1}+x^{(1)},...,a_{n}+x^{(n)})-f(a_{1},...,a_{n})\right\Vert \right)
^{p}\label{new}\\
&  \leq C^{p}\sum_{k=1}^{n}\frac{1}{p_{k}}\int_{B_{X_{k}^{\ast}}}\left(
\left\vert b^{(k)}\right\vert \left\vert \varphi(x^{(k)})\right\vert \right)
^{p_{k}}d\overline{\mu_{k}}.\nonumber
\end{align}
Let $x^{(1)},...,x^{(n)}$ and $b^{(1)},...,b^{(n)}\neq0$ be given and, for
$k=1,...,n,$ define
\[
\tau_{k}:=\left(  \int_{B_{X_{k}^{\ast}}}\left(  \left\vert b^{(k)}\right\vert
\left\vert \varphi(x^{(k)})\right\vert \right)  ^{p_{k}}d\overline{\mu_{k}%
}\right)  ^{1/p_{k}}.
\]
If $\tau_{k}=0$ for every $k$ then, from (\ref{new}) we conclude that
\[
\left(  \left\vert b^{(1)}...b^{(n)}\right\vert \left\Vert f(a_{1}%
+x^{(1)},...,a_{n}+x^{(n)})-f(a_{1},...,a_{n})\right\Vert \right)  ^{p}=0
\]
and we obtain (\ref{domGG}), as planned. Let us now suppose that $\tau_{j}$ is
not zero for some $j\in\{1,...,n\}$. Consider
\[
V=\{j\in\{1,..,n\};\tau_{j}\neq0\}
\]
and for each $\beta>0$ define%
\[
\vartheta_{\beta,j}=\left\{
\begin{array}
[c]{c}%
\left(  \tau_{j}\beta^{\frac{1}{pp_{j}}}\right)  ^{-1}\text{ if }j\in V\\
1\text{ if }j\notin V.
\end{array}
\right.
\]
So, from (\ref{new}), we have%
\begin{align*}
\frac{1}{p}f(\vartheta_{\beta,k}b^{(k)},x^{(k)})_{k=1}^{n}  &  \leq C^{p}%
\sum_{k=1}^{n}\frac{1}{p_{k}}\int_{B_{X_{k}^{\ast}}}\left(  \left\vert
\vartheta_{\beta,k}b^{(k)}\right\vert \left\vert \varphi(x^{(k)})\right\vert
\right)  ^{p_{k}}d\overline{\mu_{k}}\\
&  \leq C^{p}\sum_{k\in V}\frac{1}{p_{k}}\vartheta_{\beta,k}^{p_{k}}%
\int_{B_{X_{k}^{\ast}}}\left(  \left\vert b^{(k)}\right\vert \left\vert
\varphi(x^{(k)})\right\vert \right)  ^{p_{k}}d\overline{\mu_{k}}\\
&  \leq C^{p}\sum_{k\in V}\frac{1}{p_{k}}\left(  \tau_{k}\beta^{\frac
{1}{pp_{k}}}\right)  ^{-p_{k}}\tau_{k}^{p_{k}}\\
&  =C^{p}\sum_{k\in V}\frac{1}{p_{k}}\frac{1}{\beta^{\frac{1}{p}}}\\
&  \leq\frac{C^{p}}{p}\frac{1}{\beta^{\frac{1}{p}}}.
\end{align*}
Hence
\[
\vartheta_{\beta,1}^{p}...\vartheta_{\beta,n}^{p}\frac{1}{p}f(b^{(k)}%
,x^{(k)})_{k=1}^{n}\leq\frac{C^{p}}{p}\frac{1}{\beta^{1/p}}%
\]
and we have%
\begin{align}
f(b^{(k)},x^{(k)})_{k=1}^{n}  &  \leq C^{p}\beta^{-1/p}\left(  \vartheta
_{\beta,1}^{p}...\vartheta_{\beta,n}^{p}\right)  ^{-1}\label{wwss}\\
&  =C^{p}\beta^{-1/p}%
{\textstyle\prod\nolimits_{j\in V}}
\left(  \tau_{j}\beta^{\frac{1}{pp_{j}}}\right)  ^{p}\nonumber\\
&  =C^{p}\beta^{\left(
{\textstyle\sum\nolimits_{j\in V}}
1/p_{j}\right)  -1/p}%
{\textstyle\prod\nolimits_{j\in V}}
\tau_{j}^{p}.\nonumber
\end{align}
If $V\neq\{1,...,n\}$, then
\[
\frac{1}{p}-%
{\displaystyle\sum\nolimits_{j\in V}}
\frac{1}{p_{j}}>0.
\]
Letting $\beta\rightarrow\infty$ in (\ref{wwss}) we get%
\[
f(b^{(k)},x^{(k)})_{k=1}^{n}=0
\]
and we again reach (\ref{domGG}). If $V=\{1,...,n\}$, from (\ref{wwss}) we
conclude the proof, since
\[
\left(  \left\vert b^{(1)}...b^{(n)}\right\vert \left\Vert f(a_{1}%
+x^{(1)},...,a_{n}+x^{(n)})-f(a_{1},...,a_{n})\right\Vert \right)
^{p}=f(b^{(k)},x^{(k)})_{k=1}^{n}\leq C^{p}%
{\displaystyle\prod\limits_{j=1}^{n}}
\tau_{j}^{p}.
\]

\end{proof}

Note that inequality (\ref{qww}) seems to arise an idea of weighted
summability. We interpret as each $x_{j}^{(k)}$ has a \textquotedblleft
weight\textquotedblright\ $b_{j}^{(k)}$ and in this context the respective
sum
\[
\left\Vert f(a_{1}+x_{j}^{(1)},...,a_{n}+x_{j}^{(n)})-f(a_{1},...,a_{n}%
)\right\Vert
\]
inherits a weight $\left\vert b_{j}^{(1)}...b_{j}^{(n)}\right\vert $. It is
easy to note that if $f$ is $n$-linear and $a_{1}=...=a_{n}=0,$ then
inequality (\ref{qww}) coincides with the usual non-weighted inequality. So,
the concept of weighted summability can be viewed as a natural extension of
the multilinear concept to nonlinear (non-multilinear) maps.

\bigskip

\textbf{Acknowledgement. }The authors thank G. Botelho and P. Rueda for
helpful conversations on the topics of this paper. The authors also thank the
referee for important suggestions.

\end{document}